\documentclass{elsarticle}

\usepackage{amsfonts}
\usepackage{amsthm}
\usepackage[ruled,linesnumbered]{algorithm2e}

\newtheorem{thm}{Theorem}[section]
\newtheorem{exm}[thm]{Example}

\newtheorem{lem}[thm]{Lemma}

\newtheorem{prop}[thm]{Proposition}
\newtheorem{cor}[thm]{Corollary}

\newcommand{\abs}[1]{\left\vert#1\right\vert}

\newcommand{\conv}{\texttt{Conv}}

\newcommand{\diag}{\texttt{diag}}

\newcommand{\norm}[1]{\parallel\! #1\! \parallel}

\newcommand{\rank}{\texttt{rank}}
\newcommand{\seq}[1]{\left<#1\right>}
\newcommand{\set}[1]{\left\{#1\right\}}

\newcommand{\sgn}{\texttt{sgn}}

\newcommand{\tr}{\texttt{Tr}}

\newcommand{\al}{\alpha}

\newcommand{\be}{\beta}
\newcommand{\ga}{\gamma}
\newcommand{\Gam}{\Gamma}

\newcommand{\la}{\lambda}

\newcommand{\p}{\prime}

\newcommand{\sig}{\sigma}

\newcommand{\tea}{\theta}

\newcommand{\A}{\mathcal{A}}
\newcommand{\B}{\mathcal{B}}
\newcommand{\C}{\mathcal{C}}

\newcommand{\X}{\mathcal{X}}
\newcommand{\Y}{\mathcal{Y}}

\newcommand{\J}{\mathcal{J}}

\newcommand{\cH}{\mathcal{H}}
\newcommand{\cP}{\mathcal{P}}

\newcommand{\cK}{\mathcal{K}}
\newcommand{\caL}{\mathcal{L}}

\newcommand{\T}{\mathcal{T}}

\newcommand{\bfa}{\textbf{a}}
\newcommand{\bfb}{\textbf{b}}
\newcommand{\bfc}{\textbf{c}}

\newcommand{\bv}{\textbf{v}}

\newcommand{\bx}{\textbf{x}}
\newcommand{\pbx}{\textbf{x}^{\p}}

\newcommand{\by}{\textbf{y}}
\newcommand{\bz}{\textbf{z}}

\newcommand{\PP}{\mathbb{P}}
\newcommand{\R}{\mathbb{R}}
\newcommand{\CC}{\mathbb{C}}

\newcommand{\II}{\mathbb{I}}

\newcommand{\mnrt}{$m$th order $n$-dimensional real tensor }
\newcommand{\mnrts}{$m$th order $n$-dimensional real tensors }

\newcommand{\beq}{\begin{equation}}
\newcommand{\eeq}{\end{equation}}
\newcommand{\bey}{\begin{eqnarray}}
\newcommand{\eey}{\end{eqnarray}}
\newcommand{\beyy}{\begin{eqnarray*}}
\newcommand{\eeyy}{\end{eqnarray*}}

\begin{document}
\begin{frontmatter}
\title{ Grassmann Tensors and their Applications in Mutliview Geometry
\footnote{The first two authors contributed  equally to this paper.
This work is supported by National Natural Science Foundation of China under Grant No. 62101400. }}

\author[add1]{Changqing Xu}
\ead{cqxurichard@usts.edu.cn}
\author[add2]{Kaijie Xu}
\ead{kaijiexu@mail.xidian.edu.cn}
\author[add1]{Jun Wang}
\ead{2844832029@qq.com}
\author[add1]{Jingxuan Bai}
\ead{954686364@qq.com}
\address[add1]{School of Mathematical Sciences, Suzhou University of Science and Technology, Suzhou,China}
\address[add2]{School of Electronic Engineering, Xidian University, Xi'an, China}

\begin{abstract}
In this paper, we introduce the Grassmann tensor by tensor product of vectors and some basic terminology in tensor theory. Some basic properties of the Grassmann 
tensors are investigated and the tensor language is used to rewrite some relations and correspondences in the mutliview geometry. Finally we show that a polytope in 
the Euclidean space $\R^{n}$ can also be concisely expressed as the Grassmann tensor generated by its vertices.  
\end{abstract}

\end{frontmatter}

\section{Introduction}
\setcounter{equation}{0}

Tensors play a very important role in deep learning and machine learning\cite{SLFH2017}. They are at the heart of algorithms such as convolutional neural networks (CNNs) and recurrent neural networks (RNNs) which are widely used in many fields e.g. image recognition\cite{BS2020}, natural language processing (NLP)\cite{CMPC2017}, and 
time series analysis (TSA)\cite{RLR2013}. \\  
\indent A tensor is a multi-dimensional array of numerical values used to describe high dimensional data such as the physical state or properties of a material in physics and mechanics \cite{KS1975, KS1977, Hiki1981}.  In image analysis, we usually use a matrix to capture the intensity of light along with the spatial coordinates of the image to illustrate a grayscale image, and a third-order tensor to represent a batch of grayscale images with two spatial modes and one mode indexing different images in the batch. A color image can be algebraically described by a 3-order tensor identical to a triple of matrices $(R,G,B)$, each representing the color image in one channel. In this way, a color video can 
be represented by a fourth-order tensor $\A=(A_{ijkl})\in\R^{m\times n\times 3\times p}$ where the discretized video flow is sampled into a sequence of $p$ frames with each 
a color image of size $m\times n$. In general, the more indices a tensor has, the more degrees of freedom it represents. Thus tensors can compactly represent very high-dimensional data.  Tensors also make many operations easier to perform than matrices. For example, they are easily manipulated using automatic differentiation software like TensorFlow and ByTorch, which can automatically compute the gradient of a tensor with respect to any other tensor. \\
\indent  Deep learning models are composed of layers each of which takes some input data and produces the output after some transformations. There are many types of layers, including convolutional layers, pooling layers and recurrent layers, which are fully connected and parameterized by tensors. Tensors are also used to represent the weights of neural networks. A weight tensor is simply a tensor that is used as a parameter in a layer. When we train a neural network, we need to optimize the values of these weight tensors to minimize some loss functions.\\
\indent Tensors have already found tremendous applications in image analysis and signal processing\cite{AGTL2009, BBT2009,MSL2013,SHSW2019}, facial recognitions\cite{HKB2006} and computer vision\cite{HV2008, Hey1998, Hey2000, BBT2009}, including object and motion recognition and video understanding since 1997\cite{Quan1997}.  Quan and Kanade\cite{QK1997} and 
Faugeras et al. \cite{Faug2000} and Quan \cite{Quan2001} studied projections from $\PP^{2}$ to $\PP^{1}$, solved the reconstruction problem by trifocal tensors and 
presented two possible reconstructions in the case.  Faugeras et al. \cite{Faug2000} investigated the self-calibration of a camera undergoing motion in a plane. 
Hartley and Vidal \cite{HV2008} use tensors to compute non-rigid structure and motion under perspective projection.  The bifocal, trifocal and quadrifocal tensor have been established to reconstruct a 3D scene from its projection respectively in two, three or four images (\cite{HZ2004, Hart1998a, Hey1998, HS2009, HKB2006}).\\
\indent  There are some already known research work to reconstruct the scene points by tensors\cite{Hey2000}, including the trifocal
tensors\cite{AO2014,AST2013,AT2009,AT2010} and the quadrifocal tensors\cite{SW2000}.  Wolf and Shashua\cite{WS2002} 
investigated the projections from $\PP^{n}$ to $\PP^{2}$ to  analyze several different problems in dynamic scene, but no general way of defining such tensors. Hartley and 
Schaffalitzky\cite{HS2009} initialized the Grassmann tensors for the unification of different form of the view tensors and the algorithms for estimating the projection matrices. They show that the projection matrices can be determined uniquely by a Grassmann tensor up to projective equivalence. However, the computations involved to extract the projections and the reconstruction are still limited to matrix techniques. \\
\indent In this paper, we introduce the Grassmann (Pl\"{u}cker) tensor by the anti-symmetric operation on the tensor product of vectors, also introduced are some basic 
terminology and notations of tensor theory. We investigate properties of Grassmann tensor and then use Grassmann tensors to simplify some relations and correspondences 
in mutliview geometry. We show in the end of the paper that a polytope in Euclidean space $\R^{n}$ can be concisely expressed as the Grassmann tensor generated by its
vertices. \\
\indent Recall that a tensor is a multi-way array which is also called a hypermatrix.  A matrix is a 2-order tensor with two modes: row and column, and an $m$-order tensor 
has $m$ modes with entries addressed by $m$ indices. A tensor can also be regarded as a multilinear mapping. For more detail in tensor theory we refer the reader to 
\cite{QCC2018}.\\
\indent  We denote by $[m]$ the set $\set{1,2,\ldots,m}$ for any positive integer $m$ and by $\cP_{m}$ the set of permutations on $[m]$.  Throughout the paper, we use 
the tensorial notation, i.e., tensors of order 0 (scalars) are denoted by means of italic type letters $a,b,x,y$ and some Greek letters $\la,\mu$ etc., tensors of order 1 (vectors) 
by means of boldface italic letters $\bx,\by, \bz$ and Greek letters $\al, \be, \ga$, tensors of order two (matrices) by capital boldface letters $A, B, X,Y, M$, and tensors of higher orders by curlicue letters $\A,\B, \X,\Y, \cdots $.  An $m$-order tensor $\A$ is of size $\II_{m}:=I_{1}\times\ldots \times I_{m}$ if the dimension of the $k$-mode of $\A$  is $I_{k}$
for $k\in[m]$. Denote by $\T_{\II}$ for the set of all real tensors of size $\II_{m}$.  An $m$-order tensor $\A$ is called an \mnrt  if $I_1=I_2=\ldots=I_m:=n$ for some positive integer $m$.  The set of all \mnrts is denoted by $\T_{m,n}$.  For our convenience, we denote  
\[ S(m,n) = \set{(i_1,i_2,\ldots, i_m): i_k\in [n], \forall k\in [m] } \]
which usually serves as the index set of a tensor in $\T_{m;n}$.  An \mnrt  $\A\in \T_{m;n}$ is an $m$-array whose entries are indexed by $\sig:=(i_1,i_2,\ldots, i_m)\in S(m, n)$.  
We sometimes denote $A_{i_1i_2\ldots i_m}$ by $A_{\sig}$ for $\sig=(i_1,\ldots, i_m)$.  A tensor $\A\in\T_{m;n}$ is called \emph{symmetric} if each of its entries $A_{\sig}$ is 
invariant under any permutation of its indices, i.e., 
\[ A_{\sig}=A_{\tau(\sig)} \quad  \forall  \tau\in \cP_m, \sig\in S (m,n). \]
A symmetric tensor $\A\in\T_{m;n}$ is uniquely associated with an $m$-order $n$-variate homogeneous polynomial 
\beq\label{eq1-1}
f_{\A}(\bx)=\A\textbf{x}^m=\sum\limits_{i_1,i_2,...,i_m}A_{i_1,i_2,...,i_m}x_{i_1}x_{i_2}\ldots x_{i_m}
\eeq
$\A$ is called \emph{positive semidefinite} (\emph{positive definite}) if $f_{\A}(\bx)=\A\bx^m > 0$($\ge 0$) for all nonzero vector $\bx\in\R^n$. \\
\indent  We need some preparations on the multiplications of tensors and the related terminology and notations before we introduce the Grassmann tensor. There are mainly three kind of multiplications of tensors:  the outer product (also called tensor product), contractive (mode) product, and the t-product defined on 3-order tensors. We will introduce the first two in this paper and generalize the outer-product of two tensors. For the t-product of two 3-order tensors, we refer the reader to \cite{Lu2018, LFCL2018, TM2018}. \\  
\indent Let $p, q$ be positive integers and $\A\in\T_{p;n}, \B\in\T_{q;n}$. Let $\tea:=\set{\tea_{1}, \tea_{2},\ldots, \tea_{p}}$ be a nonempty subset of $[m]$ with complement 
$\tea^{c}:=\set{\phi_{1},\phi_{2},\ldots,\phi_{q}}$, both ordered increasingly. The \emph{outer product} of  tensors $\A,\B$ along $(\tea,\tea^{c})$, which is denoted by 
$\A\times_{(\tea,\tea^{c})} \B$ or simply $\A\times_{\tea} \B$, is defined as  
\beq\label{eq1-2}
\C=\A\times_{\tea}\B = (C_{\sig}), \mbox{with} ~ C_{\sig}=A_{\tea_{1}\tea_{2}\ldots\tea_{p}}B_{\phi_{1}\phi_{2}\ldots\phi_{q}}
\eeq
where $m:=p+q, \sig=(i_1,i_2,\ldots, i_m)$.  For $\tea:=[p]$, we denote $\A\times_{\tea} \B$ simply by $\A\times \B$.\\
\indent  The outer product of tensors satisfies the law of the association, i.e., 
\beq\label{eq1-3}
(\A\times \B)\times \C =\A\times (\B\times \C)
\eeq 
for any tensors $\A,\B$ and $\C$, and thus applies to any number of tensors.  A special case is when all tensors involved are vectors, i.e.,   
\beq\label{eq:rk1t}
\X = \bx^{(1)}\times \bx^{(2)}\times\cdots\times \bx^{(m)}, \quad \bx^{(k)}\in\C^{n}.
\eeq
$\X$ is a rank-one \mnrt  and is denoted $\bx^{m}$ when all $\bx^{(k)}$ are identical (to $\bx\in\C^{n}$).  It is shown \cite{CGLM2006} that every \mnrt  can be written as 
the sum of some rank-one tensors.\\
\indent  Given a set $S$ with cardinality $\abs{S}=n$ and a positive number $q\in [n]$. We denote by $S^{[q]}$ the set of all $q$-sets of $S$, i.e., 
$S^{[q]} = \set{ T\subset S\colon  \abs{T}=q  }$.  For example, if $S=[4]=\set{1,2,3,4}$ and $q=2$, we have 
\[ [4]^{[2]} =  \set{\set{1,2},\set{1,3},\set{1,4},\set{2,3},\set{2,4},\set{3,4}}. \] 
 
\begin{exm}\label{exm2-1}
The outer product of two matrices is another special case of (\ref{eq1-2}).  Let $A\in\C^{m_{1}\times n_{1}},B\in\C^{m_{2}\times n_{2}}$, and let 
$\tea := \{s, t\} \subset [4], s<t, \{p, q\} =\tea^c, p < q$.  Then $A \times_{\tea} B$ is a 4-order tensor satisfying
\beq\label{eq1-5}
(\A\times_{\tea}B)_{i_1i_2i_3i_4}=A_{i_s i_t}B_{i_p i_q}
\eeq
There are six outer-products corresponding resp. to six elements in $[4]^{[2]}$, i.e., $A\times_{\tea} B$ with $\tea\in\set{\set{1,2},\set{1,3},\set{1,4},\set{2,3},\set{2,4},\set{3,4}}$, 
When $A=B\in\R^{n\times n}$, we have $A\times_{\tea} A = A\times_{\tea^{c}} A$.  
\end{exm} 

\indent For any matrix $A\in\R^{p\times q}$, we can also generate different kind of $2k$-order tensor $A^{[k]}:=\overbrace{A\times A\times \cdots \times A}^{k}$. For example, 
we can define $\A=(A_{i_{1}\ldots i_{k};j_{1}\ldots j_{k}})$ as 
\beq\label{eq1-6}  
A_{i_{1}i_{2}\ldots i_{k}; j_{1}j_{2}\ldots j_{k}} =  \prod\limits_{s=1}^{k} A_{i_{s}j_{s}}  
\eeq
which is a \emph{paired symmetric tensor} (e.g. \cite{HQ2018}) since for any permutation $\al\in\cP_{3}$ we have 
\[
A_{i_{1}i_{2}\ldots i_{k}; j_{1}j_{2}\ldots j_{k}} = A_{i_{\al(1)}i_{\al(2)}\ldots i_{\al(k)}; j_{\al(1)}j_{\al(2)}\ldots j_{\al(k)}} 
\] 
\indent  Given any $2k$-order tensor $\A=(A_{i_{1}\ldots i_{k};j_{1}\ldots j_{k}})$ and an $k$-order tensor $\B$. The contractive product of $\A$ with $\B$, denoted $\A\B$, is defined as 
\beq\label{eq1-7}  
(\A\B)_{i_{1}i_{2}\ldots i_{k}} = \sum\limits_{j_{1},j_{2},\ldots,j_{k}} A_{i_{1}i_{2}\ldots i_{k};j_{1}j_{2}\ldots j_{k}} B_{j_{1}j_{2}\ldots j_{k}} 
\eeq
\indent  Given a tensor $\A\in \T_{m;n}$ and matrix $B\in\R^{n\times p}$. The contractive product of $\A$ by $B$ along the $k$-mode, denoted $\A\times_{(k)} B$ ($\forall k\in [m]$), is defined by
\beq\label{eq1-8}
(\A\times_{(k)} B)_{i_1\ldots i_{k-1} i_{k} i_{k+1}\ldots i_n} = \sum_{j=1}^n  A_{i_1,\ldots i_{k-1} j i_{k+1}\ldots i_n}B_{j i_k}
\eeq
We write $\A B$ for $\A\times_{(m)} B$. Note that it reduces to a matrix product for $m=2$, i.e., $\A\times_{(2)} B=AB, A\times_{(1)}B=A^{\top}B$.\\   
\indent  The contractive product along one mode can be generalized to two modes resulting in an $(m-2)$th order tensor, and it can preserve or compress (other than expand as a sequence of outer product) the tensors and is useful in many aspects. For example, the homogeneous polynomial $f(\bx):=\A\bx^m$ defined by (\ref{eq1-1}) can be regarded as the contractive product of $\A$ with the rank-one tensor $\bx^{m}$, and $\A\bx^{m-1}$ is employed to define various eigen-pairs (see e.g. \cite{QCC2018}). \\ 

\indent Now we let $\bv_{1},\bv_{2},\cdots, \bv_{m}\in \CC^{n}$ with each $\bv_{j}\neq 0$, and define the linear map
\beq\label{eq1-9}  
\caL :=   \sum\limits_{\sig\in \cP_{m}} (-1)^{\tau(\sig)} \sig 
\eeq
on $\T_{m;n}$ as  
\beq\label{eq1-10} 
\bv_{1}\wedge\cdots\wedge \bv_{m} =\caL(\bv_{1}\times \bv_{2}\times \cdots\times \bv_{m})  
\eeq
We call the tensor $\bv_{1}\wedge\cdots\wedge \bv_{m}$ defined by (\ref{eq1-10}) the \emph{Grassmann} (or \emph{Pl{\"u}cker}) tensor associated with 
$\set{\bv_{k}}_{k=1}^{m}$ and denote it by $\cP[\bv_{1},\ldots,\bv_{m}]$.  Note that $\caL^{2}=\caL$ and $\cP[\bv_{1},\ldots,\bv_{m}]\neq 0$ if and 
only if $\bv_{1},\cdots,\bv_{m}$ are linearly independent \cite{XX2022}.  For $m=2$, $\cP[\bv_{1},\bv_{2}]\in\R^{n\times n}$ ($n\ge 2$) is an anti-symmetric 
matrix of rank 2 when $\bv_{1}, \bv_{2}\in \R^{n}$ are linearly independent. \\
\indent  It is known\cite{XX2022} that a Grassmann tensor is an anti-symmetric tensor. Hartley and Schaffalitzky\cite{HS2009} first introduced the Grassmann tensors in 
the context of multiview geometry and used Grassmann tensors to extend the fundamental matrices (bifocal tensors) to higher order multi-focal tensors e.g. trifocus tensor 
and quadrifocus tensor to establish the relationships between sets of corresponding subspaces in various views.\\

\vskip 5pt

\section{ Grassmann tensors and their properties}
\setcounter{equation}{0}
 
Let $A=[\bfa^{1},\bfa^{2},\ldots, \bfa^{m}]\in \R^{n\times m}$ be a matrix with $\bfa^{1},\ldots, \bfa^{m}\in\R^{n}$ being its columns and $m\le n$.  
Denote $\cP:=\bfa^{1}\wedge\cdots\wedge\bfa^{m}$.  Then $\cP\in\T_{m;n}$ is a \mnrt.  We have 
\begin{lem}\label{le2-1}
$\cP\bx^{m}\equiv 0$ for all $\bx\in\R^{n}$.
\end{lem}
\begin{proof}
By definition, we have 
\beyy
\cP\bx^{m} 
&=&  \left( \sum\limits_{\sig} \sgn(\sig) \bfa^{i_{1}}\wedge \bfa^{i_{2}}\wedge\cdots \wedge \bfa^{i_{m}} \right) \bx^{m} \\
&=&  \left[ \sum\limits_{\sig} \sgn(\sig) \left(\bx^{\top}\bfa^{i_{1}}\right)\wedge\cdots \wedge \left(\bx^{\top}\bfa^{i_{m}}\right) \right] \\ 
&=& \prod\limits_{k=1}^{m}\left(\bx^{\top}\bfa^{k}\right)  \left(\sum\limits_{\sig} \sgn(\sig) \right)\\  
&=& 0.
\eeyy
Here $\sig=(i_{1},i_{2},\ldots,i_{m})\in \cP_{m}$ is any possible permutation on $[m]$.   
\end{proof}
\indent Lemma \ref{le2-1} shows that each anti-symmetric tensor corresponds to a zero polynomial, which is obvious in the matrix case. \\

\begin{thm}\label{th2-2}
Let $A=[\bfa^{1},\bfa^{2},\ldots,\bfa^{m}]\in\R^{n\times m}, \A=\bfa^{1}\wedge \bfa^{2}\wedge \ldots\wedge\bfa^{m}$. Then $\A=(A_{i_{1}\ldots i_{m}})\in\T_{m;n}$ with  
\beq\label{eq2-1}
A_{i_{1}\ldots i_{m}} =  \det A[i_{1}, \ldots, i_{m}|\colon]
\eeq
where $A[i_{1}, \ldots, i_{m}|\colon ]$ denotes the $m\times m$ submatrix of $A$ consisting of the $(i_{1}, \ldots, i_{m})$-rows of $A$. 
\end{thm}
\begin{proof}
For any $\eta:=(i_{1}, i_{2},\ldots, i_{m})\in S(m,n)$, we denote $r=\abs{\eta}$ for the number of distinct elements in set $\set{i_{1}, i_{2},\ldots, i_{m}}$. Then we have 
\beyy
A_{i_{1}i_{2}\ldots i_{m}} = A_{\eta} 
&=&(\bfa^{1}\wedge \bfa^{2}\wedge \ldots\wedge\bfa^{m})_{\eta} \\
&=&[\caL(\bfa^{1}\times\bfa^{2}\times\ldots\times\bfa^{m})]_{\eta} \\
&=&  \sum\limits_{\sig}\sgn(\sig) a_{i_{1}\sig(1)} a_{i_{2}\sig(2)} \ldots a_{i_{m}\sig(m)}\\
&=&   \det A[i_{1}, \ldots, i_{m}|\colon]   
\eeyy
\end{proof}
\indent  Given an \mnrt $\A\in\T_{m;n}$ and $\kappa:=(P_{1},\ldots, P_{m})$ where $P_{k}$ is a nonempty subset of $[n]$ for each $k$.  A subtensor of $\A$ determined by 
$\kappa$, denoted $\A[\kappa]$, is an $m$-order tensor whose elements are indexed within $P_{1}\times \cdots \times P_{m}$.  If $P_{1}= \ldots=P_{m}=S$ with 
$\abs{S}=r$, we denote $\A[\kappa]$ by $\A[S]$ and call it an \emph{$r$-dimensional principal subtensor}, i.e., $\A[S]\in\T_{m;r}$.  For $S=[r]$, $\A[S]$ is called a \emph{leading principal subtensor} (abbrev. LPS). \\
\indent  Now we let $\A=\bfa^{1}\wedge \bfa^{2}\wedge \ldots\wedge\bfa^{m}$.  By Theorem \ref{th2-2}, we see that $\A=0$ if and only if vectors $\bfa^{1},\bfa^{2},\ldots,\bfa^{m}$ are linearly dependent.  When $\bfa^{1},\bfa^{2},\ldots,\bfa^{m}$ are linearly independent, we have  

\begin{cor}\label{co2-3}
Let $A=[\bfa^{1},\bfa^{2},\ldots,\bfa^{m}]\in\R^{n\times m}$ with $\rank(A)=m\le n$ and $\A=\bfa^{1}\wedge \bfa^{2}\wedge \ldots\wedge\bfa^{m}$.  Then 
for any $\sig:=(i_{1}, i_{2},\ldots, i_{m})\in S(m,n)$,  $\A[\sig] \neq 0$ if and only if  $\abs{\sig}=m$ and $\det A[\sig |\colon ]\neq 0$. 
\end{cor} 

\begin{cor}\label{co2-4}
Let $\bfa^{1},\bfa^{2},\ldots,\bfa^{m}\in\R^{n}$ be linearly independent and $\A=\bfa^{1}\wedge \bfa^{2}\wedge \ldots\wedge\bfa^{m}$.  Then $\A$ has 
at most $\frac{n!}{(n-m)!}$ nonzero elements.   
\end{cor} 
\begin{proof}
By Corollary \ref{co2-3}, there are at most ${n\choose m}$ nonzero $LPS$ subtensors $\A[\al]\in\T_{m;m}$ where $\abs{\al}=m$ and $\al\in S(m,n)$ is an $m$-tuple 
chosen from $[n]$ with no repetitions and there are $m!$ nonzero elements in $\A[\al]$ if $\det(A[\al |\colon])\neq 0$.   
\end{proof}
\indent It follows from Corollary \ref{co2-3} that 
\begin{cor}\label{co2-5}    
\beq\label{eq2-2}
\A=\bfa^{1}\wedge \bfa^{2}\wedge \ldots\wedge\bfa^{n} = \det(A) \cH
\eeq
where $\cH=  (H_{\sig})\in\T_{n;n}$ with $H_{\sig} = \sgn(\sig)$ being the \emph{generalized sign function} defined as 
\[
\sgn(\sig) = \left\{ \begin{array}{cl}  1,  & \texttt{ if } \sig\in E_{n},\\  -1,  & \texttt{ if } \sig\in O_{n},\\  0,  & \texttt{ otherwise. } \end{array} \right.   
\]
Here $E_{n}$ and $O_{n}$ denote respectively the set of even and odd permutations on $[n]$. 
\end{cor} 
\indent As a special case of Corollary \ref{co2-3}, we have 
\begin{cor}\label{co2-6}
Let $\al,\be\in\R^{2}$. Then   
\beq\label{eq2-3}
\cP[\al,\be] = \det[\al,\be] E_{2}
\eeq
where $E_{2}$ is the elementary antisymmetric matrix defined as
\beq\label{eq2-4} 
 E_{2} = \left[\begin{array}{cc} 0&1\\ -1& 0\end{array}\right]   
\eeq  
\end{cor}

\indent Given a tensor $\A\in\T_{m;n}$ and a vector $\bfb\in\R^{n}$. The \emph{wedge product} of $\A$ with $\bfb$, denoted $\A\wedge \bfb$, is defined by 
\beq\label{eq2-5}
\A\wedge\bfb = \sum\limits_{k=1}^{m+1} (-1)^{k-1} \A\times_{\hat{k}} \bfb  
\eeq
where $\A\times_{\hat{k}} \bfb\in\T_{m+1;n}$ is the outer-product of $\A$ and $\bfb$ with the $k$-mode assigned to $\bfb$.  We will show in the following that 
$\A\wedge \bfc = \bfa\wedge \bfb\wedge \bfc$ if $\A=\bfa\wedge \bfb$ for some $\bfa,\bfb\in\R^{n}$.   

\begin{exm}\label{exm2-5}
Let $\bfa\in\R^{n}$ be a nonzero vector with $a_{i}$ as its $i$th coordinate, and $\A=(A_{ijk})=I_{n}\wedge \bfa$. Then $\A\in\T_{3;n}$ with elements 
\[
\begin{array}{ll}
A_{iii} = a_{i},                              & \forall i,\\
A_{iij} = A_{jii} =- A_{iji}=a_{j},     & \mbox{if}~ i\neq j,  \\
A_{ijk}=0,                                    & \mbox{if}~  i, j, k ~\mbox{are distinct}.  
\end{array} 
\]
Thus we have $\A\bx^{3}=\seq{\bx,\bx}\seq{\bfa, \bx} = \sum\limits_{i=1}^{n} a_{i}x_{i}^{3}$. 
\end{exm}
\indent  The following proposition offers the connection of the matrix-vector wedge and the vector wedges as well as the association law on wedge products of vectors. 
\begin{prop}\label{pro2-6}
Let $\bfa,\bfb,\bfc\in \R^{n}$ be any vectors with $n\ge 3$. Then
\beq\label{eq2-6}
(\bfa\wedge \bfb)\wedge \bfc = \bfa\wedge (\bfb\wedge \bfc) = \bfa\wedge \bfb\wedge \bfc
\eeq
\end{prop} 
\begin{proof}
We need only to show
\beq\label{eq2-7}
(\bfa\wedge \bfb)\wedge \bfc = \bfa\wedge \bfb\wedge \bfc
\eeq
Denote $A=(a_{ij})=\bfa\wedge \bfb$.  By (\ref{eq2-5}), the left hand side of (\ref{eq2-7}) can be rewritten equivalently as  
\beyy 
A\wedge \bfc  
&=& (\bfa\times \bfb - \bfb\times \bfa)\times_{\hat{1}} \bfc - (\bfa\times \bfb - \bfb\times \bfa)\times_{\hat{2}}\bfc +(\bfa\times \bfb - \bfb\times \bfa)\times_{\hat{3}}\bfc\\
&=& \bfc\times\bfa\times\bfb - \bfc\times\bfb\times \bfa - \bfa\times\bfc\times\bfb + \bfb\times\bfc\times\bfa +\bfa\times\bfb\times \bfc - \bfb\times \bfa\times\bfc\\
&=&  \bfa\times\bfb\times\bfc + \bfb\times\bfc\times\bfa + \bfc\times\bfa\times\bfb - \bfa\times\bfc\times\bfb - \bfc\times\bfb\times \bfa - \bfb\times \bfa\times\bfc\\
&=& \bfa\wedge \bfb\wedge \bfc
\eeyy 
The equality $\bfa\wedge (\bfb\wedge \bfc) = \bfa\wedge \bfb\wedge \bfc$ can also be proved by similar arguments. 
\end{proof}

\indent  To generalize (\ref{eq2-5}) to any pair of tensors $(\A,\B)$ with $\A\in\T_{p;n}, \B\in\T_{q; n}$, we let $m:=p+q, \tea\subset [m], \abs{\tea}=p$ and $\tea^{c}$ be the 
complement of $\tea$ in $[p+q]$ ($\abs{\tea^{c}}=q$). We abuse the notation $\tea$ ($\tea^{c}$) both for the subset and the corresponding sequence ordered increasingly. 
Then $\tea\cup\tea^{c}=[m]$ is a permutation of $[m]$.  Now we define   
\beq\label{eq2-8}
\A\wedge \B = \sum\limits_{\tea\in Q_{p,m}} \sgn(\tea) \A\times_{\tea} \B  
\eeq 
where $Q_{p,m}$ is the set of $p$-sequences whose entries are chosen from $[m]$ with increasing order.  For $p=q=2$ and $A,B\in\R^{n\times n}$, (\ref{eq2-8}) implies 
\beq\label{eq2-9}
A\wedge B = A\times_{(1,2)} B - A\times_{(1,3)} B + A\times_{(1,4)} B + A\times_{(2,3)} B - A\times_{(2,4)} B +A\times_{(3,4)} B 
\eeq
It is out of our expectation when we find two more plus items than the minus ones in (\ref{eq2-9}), other than a balance between the number of positive and that of negative items in the expression of the vector wedges.  Furthermore, If we take $A=B$, then 
\beq\label{eq2-10}
 A\wedge A = 2(A\times A -A\times_{(1,3)} A +A\times_{(1,4)}A)  
\eeq
\begin{cor}\label{co2-7}
Let $A\in \R^{n\times n}$ be either a symmetric matrix or a rank-one matrix. Then $A\wedge A = 2 (A\times A)$.
\end{cor} 
\begin{proof}
First we assume that $\rank(A)=1$. Then there exist some vectors $\al,\be\in\R^{n}$ such that $A=\al\times \be$. Note that 
\[ 
(\al\times \be)\times_{(1,3)} (\al\times \be) = (\al\times \be)\times_{(1,4)} (\al\times \be) = \al\times\al \times \be\times \be,  
\]
by (\ref{eq2-9}), we get 
\beyy
A\wedge A &=& (\al\times \be)\wedge (\al\times \be) \\
                  &=& 2 [(\al\times \be)\times (\al\times \be) - (\al\times \be)\times_{(1,3)} (\al\times \be) +(\al\times \be)\times_{(1,4)} (\al\times \be)]\\
                  &=& 2(\al\times\be\times\al\times\be) =2A\times A. 
\eeyy
If $A$ is symmetric, then $A$ can be written as $A=\sum\limits_{i=1}^{r} \al_{i}\times \al_{i}$ for some $\al_{i}\in\R^{n}$. Thus we have by 
\beyy
A\wedge A &=& \sum\limits_{i,j}(\al_{i}\times \al_{i})\wedge (\al_{j}\times\al_{j}) \\
                  &=& 2\sum\limits_{i,j} \al_{i}\times\al_{i}\times\al_{j}\times\al_{j} =2(A\times A). 
\eeyy
The last equation comes from the expansion of $A\times A$. 
\end{proof}

\begin{cor}\label{co2-8}
For any matrix $X\in\R^{n\times n}$, we have 
\beq\label{eq2-11}
(I_{n}\wedge I_{n})X = 2 \tr(X) I_{n} 
\eeq
where the tensor-matrix product is contractive on the last two modes. 
\end{cor}
\begin{proof}
We take $A$ to be the identity matrix $I_{n}$ in (\ref{eq2-10}) and denote $\J = I_{n}\wedge I_{n}$, then $\J=2 I_{n}\times I_{n} \in\T_{4;n}$ by 
Corollary \ref{co2-7} and so 
\[  (\J X)_{ij} = 2[(I_{n}\times I_{n})X]_{ij} = 2\delta_{ij}\sum\limits_{k_{1},k_{2}}\delta_{k_{1}k_{2}}X_{k_{1}k_{2}} = 2 \tr(X) \delta_{ij}, \] 
for all $i, j$.  Thus (\ref{eq2-11}) holds. 
\end{proof}
\indent We remark that tensor $\cK_{n}:=I_{n}\times_{(1,3)} I_{n}\in\T_{4;n}$ acts as an \emph{identity map} since $\cK_{n}X=X=X\cK_{n}$ for all $X\in\R^{n\times n}$  
where 
\[ (\A X)_{ij} =\sum\limits_{i^{\p},j^{\p}} A_{iji^{\p}j^{\p}} X_{i^{\p}j^{\p}},  (X\A)_{ij} = \sum\limits_{i^{\p},j^{\p}} X_{i^{\p}j^{\p}} A_{i^{\p}j^{\p}ij}. \]
$\cK_{n}$ is also called the \emph{commutation tensor} in statistics \cite{XHL2020}.  

\indent The following lemma, which is also of interest in itself and will be used to prove the main result in the next section, is to be generalized to Lemma \ref{le3-8}.\\
\begin{lem}\label{le2-10}
Let $\al,\be,\ga\in \R^{n}$ be linearly independent and $n\ge 3$. Then $\al\wedge \be, \be\wedge \ga, \ga\wedge \al$ are also linearly independent in the Grassmann 
algebra.   
\end{lem}
\begin{proof}
If $\al\wedge \be, \be\wedge \ga, \ga\wedge \al$ are linearly dependent, then there exist some scalar $\la,\mu\in\R$ such that 
\[  \al\wedge\be = \la(\be\wedge\ga) + \mu(\ga\wedge\al)  \]
which implies that $\al\wedge \be = (\la\be -\mu \al) \wedge \ga$ and thus $\al\wedge\be\wedge\ga =0$ (e.g. \cite{XX2022}). It turns out by \cite{XX2022} that $\al,\be,\ga$ 
must be linearly dependent, a contradiction to the hypothesis.   
\end{proof}

\vskip 5pt

\section{ Applications of Grassmann tensors in multiview geometry and geometry}
\setcounter{equation}{0}

Tensors can be employed to express the correspondences in computer vision. They can also be used to estimate the fundamental matrix, which is crucial in 3D
reconstruction from two-views in the vision.  A fundamental matrix $F$ can be described as the homography transforming a point $\bx$ in an image plane to a line 
$l^{\p} = F\bx$ in another plane, so    
\beq\label{eq3-1}
(\pbx)^{\top}F\bx=0
\eeq 
holds for every corresponding point pair $(\bx,\pbx)$. (\ref{eq3-1}) can alternatively be written in tensor form as 
\beq\label{eq3-2}
F\times_{1} \pbx\times_{2} \bx = 0
\eeq 
where $F$ is viewed as a second-order tensor.  Note that here all points and planes are expressed in homogeneous coordinate system, i.e., $\bx,\pbx\in \R^{3}$.  Another
expression for $F$ is through the two camera matrices. \\
\indent  The Grassmann tensor can be used to simplify some expressions in the geometry. For example, a plane determined by three points at general positions can be 
is usually expressed by a determinant equation as in the following.
\begin{prop}\label{p3-1} 
Let $X_{1},X_{2},X_{3}\in\PP^{3}$ ($X_{k}$ is 4-dimensional) be located in general positions, i.e., they are non-colinear. Then a point $X\in\R^{4}$ lies on the plane determined by $X_{1},X_{2},X_{3}$ if and only if  $\det [X,X_{1},X_{2},X_{3}] = 0$.
\end{prop}
By the Grassmann tensor, however, we can put it in a more concise form, as stated in the following theorem. 
\begin{thm}\label{th3-2}
Let $\al_{1},\al_{2},\al_{3}\in\R^{4}$ be vectors representing respectively planes $\pi_{1},\pi_{2}$ and $\pi_{3}$ located at general positions.  Then these planes meet  
at a unique point $\bx\in\R^{4}$ if and only if  $\A\bx=0$ where $\A=\al_{1}\wedge \al_{2}\wedge \al_{3}\in\T_{3;4}$. 
\end{thm}
\begin{proof}
We let $\bx\in\R^{4}$ be the intersection point of the three planes $\pi_{1}, \pi_{2}$ and $\pi_{3}$. By definition, we have 
\beq\label{eqt3-3}
\al_{i}^{\top}\bx =0, \quad i=1,2,3. 
\eeq
Then by definition we have 
\beyy
 \A\bx  &=& (\al_{1}\wedge \al_{2}\wedge \al_{3})\bx \\
           &=& (\al_{1}\times \al_{2}\times \al_{3} + \al_{2}\times \al_{3}\times \al_{1} +\al_{3}\times \al_{1}\times \al_{2} \\
           &  & -\al_{2}\times \al_{1}\times \al_{3} - \al_{3}\times \al_{2}\times \al_{1}  - \al_{1}\times \al_{3}\times \al_{2})\bx \\
           &=&(\al_{3}^{\top}\bx)\al_{1}\wedge\al_{2} + (\al_{1}^{\top}\bx)\al_{2}\wedge\al_{3} +(\al_{2}^{\top}\bx)\al_{3}\wedge\al_{1}               
\eeyy
So we have 
\beq\label{eqt3-4}
\A\bx = (\al_{3}^{\top}\bx)\al_{1}\wedge\al_{2} + (\al_{1}^{\top}\bx)\al_{2}\wedge\al_{3} +(\al_{2}^{\top}\bx)\al_{3}\wedge\al_{1} 
\eeq
Since $\al_{1},\al_{2},\al_{3}$ are linearly independent, we know that $\al_{1}\wedge \al_{2}, \al_{2}\wedge \al_{3}, \al_{3}\wedge \al_{1}$ are also linearly independent
by Lemma \ref{le2-10}.  It follows from (\ref{eqt3-4}) that $\A\bx=0$ is equivalent to (\ref{eqt3-3}). 
\end{proof}
\indent By the symmetry, we can deduce the following result which is analog to Theorem \ref{th3-2}.
\begin{cor}\label{co3-3} 
Let $\al_{1},\al_{2},\al_{3}\in\R^{4}$ be vectors representing three points in $\R^{3}$ at general positions.  Then they uniquely determine a plane $\pi\in\R^{4}$. Furthermore, 
a point $\bx\in\R^{4}$ lies on $\pi$ if and only if  $\bx$ satisfies condition $\A\bx=0$ where $\A=\al_{1}\wedge \al_{2}\wedge \al_{3}\in\T_{3;4}$. 
\end{cor}
\indent By Corollary \ref{co3-3}, we call $\A=\al_{1}\wedge \al_{2}\wedge \al_{3}$ the \emph{Grassmann tensor} or \emph{Pl{\"u}cker tensor} associated with plane $\pi$ 
which is determined by points $\al_{1},\al_{2},\al_{3}\in\R^{4}$. This is actually the extension of the \emph{Pl{\"u}cker matrix} of a line, as explained in the follows. \\
\indent  Let $L(X,Y)$ denote the line determined by  two distinct points $X,Y\in\R^{4}$ which are homogeneous represented.  The \emph{Pl{\"u}cker matrix} associated with 
$X,Y$, denoted by $P:=P[X,Y]$, is defined as 
\beq\label{eq:Pluckline1}
P:=X\wedge Y=X\times Y -Y\times X =XY^{\top} -YX^{\top}
\eeq
$P[X,Y]\in\R^{4\times 4}$ is a rank-2 anti-symmetric matrix. Different pairs $(X,Y)$ may give the same Pl{\"u}cker matrix $P$. But we do have 
\begin{thm}\label{th3-5}
Let $(X,Y)$ and $(X^{\p}, Y^{\p})$ be two point pairs with $X,Y,X^{\p},Y^{\p}\in \R^{4}$. Then $P[X,Y]=\la P[X^{\p},Y^{\p}]$ for some nonzero scalar $\la$ if and only if there is a nonsingular matrix 
$Q=(q_{ij})\in\R^{2\times 2}$ such that 
\beq\label{eq3-5}
\left\{ \begin{array}{ccc} X^{\p} &=& q_{11}X +q_{12}Y,\\ Y^{\p} &=& q_{21}X +q_{22}Y \end{array}\right.   
\eeq
\end{thm}
\begin{proof}
We let $M=[X,Y,X^{\p},Y^{\p}]$. Then $M\in \R^{4\times 4}$. We want to prove that $\rank(M)=2$. We write $M_{1}=[X,Y]$ and $M_{2}=[X^{\p},Y^{\p}]$, then we have 
$M_{k}\in \R^{4\times 2}$. If $X,Y$ are linearly dependent, then $P[X,Y]=0$, thus we may assume that $X,Y$ (and $X^{\p}, Y^{\p}$) are linearly independent. So 
$\rank(M_{k})=2$ for $k=1,2$.\\  
\indent Denote by $N(A)$ the null space of a given matrix $A$. We now show that  
\bey\label{eq3-6}
N(P[X,Y]) &=N(M_{1}^{\top})\label{eq3-6-1}  \\
N(P[X^{\p},Y^{\p}])&=N(M_{2}^{\top})\label{eq3-6-2} 
\eey 
\indent  To prove (\ref{eq3-6-1}), we let $Z\in N(P[X,Y])$. Then 
\[ 0=P[X,Y]Z = (XY^{\top}-YX^{\top})Z = (Y^{\top}Z) X - (X^{\top}Z) Y, \] 
which implies $X^{\p} Z = 0$ and  $Y^{\p} Z =0$, and thus $Z\in N(M^{\top})$. It follows that $N(P[X,Y])\subseteq N(M_{1}^{\top})$. Conversely, we can show 
$N(M_{1}^{\top})\subseteq N(P[X,Y])$ by reversing the argument.  Thus (\ref{eq3-6-1}) holds.  Similarly we can show (\ref{eq3-6-2}).  \\
\indent To show the sufficiency of (\ref{eq3-5}), we first note that 
\beq\label{eq3-7}
P[X,Y]=M_{1}DM_{1}^{\top}, P[X^{\p},Y^{\p}]=M_{2}DM_{2}^{\top}
\eeq
where $D=E_{2}$ is defined by (\ref{eq2-4}). If (\ref{eq3-5}) holds, i.e., $M_{2}=M_{1}Q$ with $Q\in\R^{2\times 2}$ nonsingular, then we have 
\beq\label{eq3-8}  
P[X^{\p},Y^{\p}] = M_{2}DM_{2}^{\top} = M_{1}QDQ^{\top}M_{1}^{\top} = M_{1}P[q_{1},q_{2}] M_{1}^{\top} 
\eeq
where $Q=[q_{1},q_{2}]$ with $q_{k}\in \R^{2}$($k=1,2$). We note that $q_{1},q_{2}$ are linearly independent due to the nonsingularity of $Q$, and thus $P[q_{1},q_{2}]$
is not zero. Furthermore, we have $P[q_{1},q_{2}]=\det(Q) D$ by simple computations. Consequently we get by (\ref{eq3-8}) 
\[  P[X^{\p},Y^{\p}] = \det(Q) M_{1}DM_{1}^{\top} =\det(Q) P[X,Y]. \]     
\indent To prove the necessity, we assume w.l.g. that $\la=1$ and $P[X,Y]=P[X^{\p},Y^{\p}]$. By (\ref{eq3-6-1}) and (\ref{eq3-6-2}), we have 
$N(M_{1}^{\top})=N(M_{2}^{\top})$. Thus 
\[ N(M^{\top})=N(M_{1}^{\top})\cap N(M_{2}^{\top}) = N(M_{1}^{\top}), \]
which implies $\rank(M)=\rank(M^{\top}) = \rank(M_{1}^{\p})=2$. Since $X,Y$ are linearly independent, there exists a matrix $Q=(q_{ij})\in\R^{2\times 2}$ such that  
$M_{2}=M_{1}Q$. Similarly we have $M_{1}=M_{2}Q^{\p}$ for some $Q^{\p}=(q^{\p}_{ij})\in\R^{2\times 2}$ by swapping $(X,Y)$ with $(X^{\p},Y^{\p})$. It follows that 
$Q$ (and also $Q^{\p}$) is invertible (nonsingular). The proof is completed. 
\end{proof}

\indent Theorem \ref{th3-5} allows us to choose an invertible matrix $Q\in\R^{2\times 2}$ such that $\set{X^{\p},Y^{\p}}$ is orthognormal, i.e. , 
\beq\label{eq3-9}
\seq{X^{\p},X^{\p}} = \seq{Y^{\p},Y^{\p}} =1, \quad \seq{X^{\p},Y^{\p}} = 0
\eeq
with $P[X,Y]=P[X^{\p},Y^{\p}]$.  In fact, we can use the Schmidt orthognormal  process to obtain $X^{\p},Y^{\p}$: 
\[
\begin{array}{lll} 
 X^{\p}  &=\la_{1} X,                          &\la_{1}:=\frac{1}{\norm{X}},\\
 \bar{Y} &=Y - \seq{X^{\p},Y} X^{\p}, & \\
  Y^{\p} &=\la_{2} \bar{Y},                 &\la_{2}:=\frac{1}{\norm{\bar{Y}}} 
\end{array}   
\] 
\indent We now show the following theorem.
\begin{thm}\label{th3-6}
Let $X,Y\in \R^{4}$ be orthognormal. Then the Pl{\"u}cker matrix $P=P[X,Y]$ is the reflection with respect to $Y$.
\end{thm}  
\begin{proof}
We first assume that $X,Y\in\R^{n}$ are orthognormal vectors, i.e., 
\[ \norm{X}=\norm{Y}=1, \seq{X,Y}=X^{\top}Y=0, \]
where the norm is $2$-norm.  Then we have 
\[ PX = (XY^{\top}-YX^{\top})X = -Y, \quad  PY = (XY^{\top}-YX^{\top})Y = X, \]
Denote $\al=X+\imath Y, \be = X -\imath Y$ where $\imath:=\sqrt{-1}$. Then $\al, \be$ are orthognormal (thus also linearly independent).  Furthermore, we can 
check easily that 
\[ P[\al,\be] = [\al, \be] K, \]
where $K=\diag(\imath, -\imath)$, i.e., $P$ maps $\al$ to $\al^{\p}=\imath \al$ and $\be$ to $\be^{\p}= -\imath \be$. 
\end{proof}

\indent  The Pl{\"u}cker matrix $P=\bfa\bfb^{\top} - \bfb \bfa^{\top}$ can be used to represent the intersection line $l$ of two planes(see e.g. \cite{HS2009}).
\begin{prop}\label{p3-7}
Let $\bfa,\bfb\in\R^{4}$ be linearly independent vectors representing two distinct planes $\pi_{1}$ and $\pi_{2}$ respectively, and $l$ be their intersection line, i.e., 
\[ l: \left\{ \begin{array}{cl}  \pi_{1}: & \bfa^{\top}\bx =0, \\  \pi_{2}: & \bfb^{\top}\bx =0.\end{array} \right. \]
Then $P=\bfa\bfb^{\top} - \bfb \bfa^{\top}$ is the Pl{\"u}cker matrix of $l$, i.e., $\bx\in l$ if and only if $P\bx=0$. Furthermore, If $Q\in\R^{4\times 4}$ is also a Pl{\"u}cker matrix of $l$, then $Q=\la P$ for some nonzero scalar $\la\in\R$.   
\end{prop}
\indent 
In order to generalize Proposition \ref{p3-7}, we first need the following lemma which is an extension of Lemma \ref{le2-10}. 
\begin{lem}\label{le3-8}
 Let vectors $\bfa^{1},\bfa^{2},\ldots, \bfa^{m}\in \R^{n}$ be linearly independent with $1< m\in [n]$.  Then the $(m-1)$-vectors in set 
\beq\label{eq3-10}
\Gam_{m-1}:=\set{\bfa^{i_{1}}\wedge \bfa^{i_{2}}\wedge\cdots\wedge\bfa^{i_{m-1}}\colon \set{i_{1},i_{2},\ldots,i_{m-1}} \subset [m].  }
\eeq  
are linearly independent.
\end{lem}  
\begin{proof}
The result is obvious for $m=1,2$, and it is also true for $m=3$ by Lemma \ref{le2-10}.  Now we assume that $m>3$ ($m\le n$). Denote 
\beq\label{eq3-11} 
\al^{(k)}=\bigwedge\limits_{j\neq k, j=1}^{m} \bfa^{j}, \quad  k=1,2,\ldots,m. 
\eeq
Then $\Gamma_{m-1}=\set{\al^{(1)}, \al^{(2)}, \cdots,\al^{(m)}}$.  Suppose there are some $\la_{1},\ldots, \la_{m}\in\R$, not all zeros, such that 
\beq\label{eq3-12} 
\la_{1}\al^{(1)} + \la_{2}\al^{(2)} + \ldots +\la_{m}\al^{(m)} =0. 
\eeq
We may assume that the last nonzero scalar in $\set{\la_{k}}$ is $\la_{m^{\p}}$ with $1< m^{\p} \le m$, i.e., 
$0\neq \la_{m^{\p}}$, and $\la_{m^{\p}+1} = \la_{m^{\p}+2} = \cdots = \la_{m} =0$ (if $m^{\p}=m$, then $\la_{m}\neq 0$).  For simplicity, we may assume that $m^{\p}=m$.
Then (\ref{eq3-12}) is equivalent to 
\beq\label{eq3-13}
\al^{(m)} = \sum\limits_{k=1}^{m-1} \la^{\p}_{k}\al^{(k)}, \quad \la^{\p}_{k}=\frac{\la_{k}}{\la_{m}}, k\in [m-1].
\eeq
By (\ref{eq3-13}), we have 
\beq\label{eq3-15}
\al^{(m)}\wedge \bfa^{m} = \sum\limits_{k=1}^{m-1} \la^{\p}_{k} (\al^{k}\wedge\bfa^{m}) = 0.
\eeq
On the other hand, we have 
\[  \al^{(m)}\wedge \bfa^{m}=\bfa^{1}\wedge \bfa^{2}\wedge\cdots\wedge\bfa^{m} \]
which is nonzero due to the linear independency of $\set{\bfa^{k}\colon k\in [m]}$  (see e.g. \cite{XX2022}), a contradiction to (\ref{eq3-15}). 
Thus we have $\la_{1}=\cdots =\la_{m}=0$, and so result holds.
\end{proof}  
\indent To state our last result, we recall that a polytope in $\R^{d}$ generated by a set of points $X:=\set{\bx^{1},\bx^{2},\cdots, \bx^{m}}\subset \R^{d}$ is defined 
as the set of the affine combinations of the points in $X$, i.e., 
\[
W=\conv(X):=\set{\sum\limits_{j=1}^{m} \la_{j}\bx^{j}\colon \sum\limits_{j=1}^{m} \la_{j}=1,  \la^{j}\ge 0\  \forall j\in [m]. }
\]
$W$ can also be represented as the bounded solution set of a finite system of linear inequalities, i.e., 
\beq\label{eq3-16} 
W =W(A, b):=\set{\bx\in\R^{d}:  A^{\top}\bx \le \bfb }  
\eeq
Here $A=[\bfa^{1}, \bfa^{2},\cdots,\bfa^{m}]\in\R^{d\times m}, \bfb\in \R^{m}$.  \\

\indent  We end the paper by the following theorem in which the Grassmann tensor is used to describe a polytope in geometry.  

\begin{thm}\label{th3-9}
 Let $\hat{W}$ be the surface (bounder) of the polytope $W$ defined by (\ref{eq3-16}) where $A=[\bfa^{1}, \bfa^{2},\cdots,\bfa^{m}]\in\R^{d\times m}, \bfb\in \R^{m}$ with 
 $\rank(A)=r >1$. Then $\bx\in \hat{W}$ if and only if  
 \beq\label{eq3-17} 
 (\bfa^{i_{1}}\wedge\bfa^{i_{2}}\wedge\cdots\wedge\bfa^{i_{r}})\bx = 0
 \eeq
 where $\set{\bfa^{i_{1}}, \bfa^{i_{2}}, \cdots,\bfa^{i_{r}}}$ is a basis of the set of vectors $X=\set{\bfa^{1}, \bfa^{2},\ldots,\bfa^{m}}$.   
\end{thm} 
\begin{proof}
We may assume without loss of generality that $r=m$ since otherwise we can replace $A$ by its submatrix $A_{1}\in\R^{d\times r}$ whose column vectors form a basis 
of $X:=\set{\bfa^{1}, \bfa^{2},\cdots,\bfa^{m}}$. Denote $\cP:=\bfa^{1}\wedge\bfa^{2}\wedge\cdots\wedge\bfa^{m}$ and $\al^{k}$ be defined by (\ref{eq3-11}). Then      
\beyy
0=\cP\bx  &= (\bfa^{1}\wedge\bfa^{2}\wedge\cdots\wedge\bfa^{m}) \bx \\
            &=\sum\limits_{k=1}^{m} (-1)^{m-k}(\bx^{\top}\bfa^{k})\al^{k}     
\eeyy 
Since $X$ is a set of linearly independent vectors, $\al^{1},\al^{2},\ldots, \al^{m}$ are also linearly independent by Lemma \ref{le3-8}. Thus we have 
$\bx^{\top}\bfa^{k}=(\bfa^{k})^{\top}\bx =0$ for all $k\in [m]$, which is equivalent to $A^{\top}\bx =0$, that is, $\bx$ is on the surface of the polytope $W$. 
The converse can be proved by reversing the arguments.       
\end{proof}

\end{document}